\documentclass{article}
\usepackage[utf8]{inputenc}
\usepackage{amsthm}
\usepackage{amssymb}
\usepackage{amsmath}
\usepackage{amsfonts}
\usepackage{float}
\usepackage{geometry}
\usepackage[mathscr]{euscript}
\usepackage{tikz}
\usepackage{caption}
\usepackage{subcaption}
\usepackage{graphicx}

\theoremstyle{plain}
\newtheorem{theorem}{Theorem}[section]
\newtheorem{lemma}[theorem]{Lemma}
\newtheorem{corollary}[theorem]{Corollary}
\newtheorem{proposition}[theorem]{Proposition}
\newtheorem{conjecture}{Conjecture}

\theoremstyle{definition}
\newtheorem{definition}[theorem]{Definition}

\graphicspath{ {./figures/} }

\newcommand{\R}{\mathbb R}
\newcommand{\f}{f}

\DeclareMathOperator{\dist}{dist}

\tikzstyle{every node}=[circle, fill=black, inner sep=1pt]

\title{New conjectures on algebraic connectivity and the Laplacian spread of graphs}

\author{Wayne Barrett\footnote{Department of Mathematics, Brigham Young University, Provo, UT, wb@mathematics.byu.edu}, Emily Evans\footnote{Department of Mathematics, Brigham Young University, Provo, UT, ejevans@mathematics.byu.edu}, H.~Tracy Hall\footnote{Hall Labs, LLC, Provo, UT, h.tracy@gmail.com}, and Mark Kempton\footnote{Department of Mathematics, Brigham Young University, Provo, UT, mkempton@mathematics.byu.edu}}
\date{}

\begin{document}

\maketitle

\begin{abstract}
    We conjecture a new lower bound on the algebraic connectivity of a graph that involves the number of vertices of high eccentricity in a graph. We prove that this lower bound implies a strengthening of the Laplacian Spread Conjecture.  We discuss further conjectures, also strengthening the Laplacian Spread Conjecture, that include a conjecture for simple graphs and a conjecture for weighted graphs.
\end{abstract}

\section{Introduction}

The algebraic connectivity of a graph is one of the most well-studied parameters in spectral graph theory.  It is defined as the second smallest eigenvalue of the combinatorial Laplacian matrix of a graph.  That is, if the Laplacian matrix $L$ of a graph $G$ has eigenvalues $0=\lambda_1(G)\leq\lambda_2(G)\leq\cdots\leq\lambda_n(G)$, the algebraic connectivity is $\lambda_2(G)$.  It is well known that this eigenvalue gives a measure for how well-connected a graph is.  Famously, $\lambda_2(G) > 0$ if and only if $G$ is connected.  Good lower bounds on $\lambda_2(G)$ have been difficult to obtain.  For example in Tables 4.1 and 4.2 of~\cite{OLDANDNEW} we find twelve upper bounds for $\lambda_2(G)$ but only four lower bounds.  Moreover, these lower bounds are typically far from sharp.

 For a graph $G$ on $n$ vertices with Laplacian eigenvalues $0=\lambda_1(G)\leq\lambda_2(G)\leq\cdots\leq\lambda_n(G)$, the Laplacian Spread Conjecture states that
 \begin{equation}
     \lambda_n(G)-\lambda_2(G) \leq n-1.
 \end{equation}
 Because of the straightforward relationship that $\lambda_i(G)+\lambda_{n+2-i}(G^c) = n \text{ for } i=2,...,n,$ where $G^c$ denotes the complement of $G$,
 this conjecture can be reformulated as stating
 \begin{equation}
     \lambda_n(G)+\lambda_n(G^c) \leq 2n-1
 \end{equation}
 or equivalently
 \begin{equation}\label{eq:lam2}
     \lambda_2(G) + \lambda_2(G^c) \geq 1.
 \end{equation}
 
 In light of the connection between $\lambda_2$ and how well-connected the graph is, the symmetric formulation of the Laplacian spread conjecture in (\ref{eq:lam2}) can be interpreted as stating that a graph and its complement cannot both be very poorly connected. 
 
  The Laplacian Spread Conjecture was recently resolved in \cite{einollahzadeh2021proof} via an ingenious argument giving very technical lower bounds on $\lambda_2(G)$ and $\lambda_2(G^c)$ simultaneously, thus implying (\ref{eq:lam2}).  The purpose of the present work is to conjecture a more simple, intuitive lower bound on the algebraic connectivity, which we now state.
 
 
 Define \[D(G) = \{v\in V(G) : \text{there is some }x \in V(G) \text{ with } d(x,v)\geq3\}. \]  Recall the \emph{eccentricity} of a vertex $v$ of a graph is $\max\{d(v,x):x\in V(G)\}$.  Thus $D(G)$ consists of all the vertices of eccentricity at least 3.

 \begin{conjecture}\label{conj:bound}
 For a graph $G$ on $n$ vertices with $k=|D(G)|/2$, we have 
 \begin{equation}\label{eq:conj_bound}
     \lambda_2(G) \geq \frac{n-k+1-\sqrt{(n-k+1)^2-4(n-2k)}}{2}.
 \end{equation}
 \end{conjecture}
 
Our motivation for considering this expression can be seen in Theorem \ref{thm:Gfamily} below and the remarks following its proof.
 
For ease of notation, let us define 
\[
\f_n(x) = \frac{n-x+1-\sqrt{(n-x+1)^2-4(n-2x)}}{2},
\] 
so that (\ref{eq:conj_bound}) can be restated as 
\[
\lambda_2(G)\geq \f_n\left(\frac{|D(G)|}{2}\right).
\]

We note that $\f_n(x)$ is real and nonnegative for $x$ in the relevant domain ($0\leq x \leq n/2$).  Note also $\f_n(x)$ is monotonically decreasing in $x$ (see Lemma \ref{lem:decr} in the appendix).  Thus, a large number of vertices with eccentricity 3 or more will correspond to a smaller bound on $\lambda_2$.  This matches the intuition that $\lambda_2$ is giving a measure for how well-connected a graph is.  Graphs with many vertices that are far from each other are more poorly connected, while graphs in which most vertices are close to each other are well connected.

One advantage of a bound involving $|D(G)|$ is that it is straightforward to prove that $|D(G)|$ and $|D(G^c)|$ cannot both be large (see Proposition \ref{prop:comp} below).  Thus our conjectured bound becomes relevant to questions concerning the algebraic connectivity of both a graph and its complement.

We will prove that Conjecture \ref{conj:bound}, if true, would give an alternate proof of the Laplacian spread conjecture.  It has been known for some time that the spread conjecture is true for graphs of diameter other than $3$.  
We will discuss how Conjecture \ref{conj:bound}, in fact, gives a stronger bound on the Laplacian spread for graphs of diameter $3$.

Closely related to the Laplacian Spread Conjecture is the question of lower bounding the quantity $\max\{\lambda_2(G),\lambda_2(G^c)\}$.  In \cite{afshari2018algebraic} it is shown that
\begin{equation}
    \max\{\lambda_2(G),\lambda_2(G^c)\} \geq \frac25
\end{equation}
for any graph.  They further point out that the self-complementary graph $P_4$ (the path on $4$ vertices) has $\lambda_2=2-\sqrt2\approx .5858$, so that no absolute lower bound greater than $2-\sqrt2$ could be hoped for to bound $\max\{\lambda_2(G),\lambda_2(G^c)\}$.  In addition, \cite{einollahzadeh2021proof} obtains the asymptotic result
\[
\max\{\lambda_2(G),\lambda_2(G^c)\} \geq 1-O(n^{-1/3}).
\]  In this paper, we will show that Conjecture \ref{conj:bound} implies that
\begin{equation}\label{eq:2rt2}
    \max\{\lambda_2(G),\lambda_2(G^c)\} \geq 2-\sqrt2
\end{equation}
so our conjecture would resolve the question of the optimal lower bound on $\max\{\lambda_2(G),\lambda_2(G^c)\}$.


The remainder of this paper is organized as follows.  In Section \ref{sec:spread} we will show that Conjecture \ref{conj:bound} implies (\ref{eq:lam2}) and (\ref{eq:2rt2}) above.  We will also see how it can strengthen the Laplacian Spread Conjecture, and compare its bound to other known bounds on $\lambda_2$.  In Section \ref{sec:motivation}, we will prove the conjecture is true for several different families of graphs.  These families will include all graphs of diameter other than 3 and 4, all trees, and some families of graphs with diameter 3 and 4 that have particularly small $\lambda_2$.  These families will also explain the motivation for Conjecture \ref{conj:bound}.  Finally, in Section \ref{sec:conjs}, we will explore some further related conjectures, including a conjecture analogous to the Laplacian Spread Conjecture for weighted graphs, and a conjecture on the optimal spread of all diameter 3 graphs.


\section{Relation to the Laplacian spread and $\max\{\lambda_2(G),\lambda_2(G^c)\}$}\label{sec:spread}

 \begin{proposition}\label{prop:comp}
 If $G$ is a graph on $n$ vertices, then
\[|D(G)|+|D(G^c)| \leq n.\]
 \end{proposition}
 \begin{proof}
 We claim that any vertex $v$ of $G$ that has eccentricity 3 or more in $G$ will have eccentricity at most 2 in $G^c$.  To this end, suppose $v$ has eccentricity 3 in $G$.  We will show that $v$ is distance at most 2 from any vertex in $G^c$.  Let $x$ be an arbitrary vertex of $G$.  If $v$ and $x$ are not adjacent in $G$, then they are adjacent in $G^c$, so assume that $v$ and $x$ are adjacent in $G$.  Let $y$ be a vertex in $G$ at distance 3 or more from $v$ in $G$.  Note that $y$ cannot be adjacent to $x$, since if it was, then since $x$ is adjacent to $v$ in $G$, then $y$ would be distance only 2 from $v$.  Thus, in $G^c$, $y$ is adjacent to both $v$ and $x$.  Thus $v$ and $x$ are within distance 2 of each other in $G^c$.  This completes the proof.
 \end{proof}
 
 \begin{theorem}\label{thm:conj_impl_spread}
 Let $G$ be a graph on $n\geq3$ vertices.  Then
 Conjecture \ref{conj:bound} implies the inequality (\ref{eq:lam2}), the symmetric formulation of the Laplacian spread conjecture.
 \end{theorem}
 \begin{proof}
 Note that Conjecture \ref{conj:bound} together with Proposition \ref{prop:comp} and Lemma \ref{lem:decr} in the appendix, with $k=|D(G)|/2$, imply that 
 \begin{align*}
 \lambda_2(G)+\lambda_2(G^c) &\geq 
 \f_n(k)+\f_n\left(\frac{n-2k}{2}\right)\\
 &=1+\frac{3n/2-\sqrt{(n-k+1)^2-4(n-2k)}-\sqrt{(\frac{n+2k+2}{2})^2-8k}}{2}\\
 &=1+g_n(k).
 \end{align*}
 where we define 
 \begin{equation}\label{eq:gnkcd}g_n(k)=\frac{3n/2-\sqrt{(n-k+1)^2-4(n-2k)}-\sqrt{(\frac{n+2k+2}{2})^2-8k}}{2}.\end{equation}
 
Thus we will be done if we can prove that $g_n(k)$ is nonnegative for $k\in[0,\frac{n}{2}]$. Observe that $g_n(0)=g_n(n/2)=0$, and that $g_n$ is concave down on the interval $k\in[0,\frac{n}{2}]$ (see Lemma~\ref{lem:concavedown} in the appendix).  This gives the desired result. 
 \end{proof}

\begin{theorem}
Let $G$ be a graph on $n\geq2$ vertices.  Then
Conjecture \ref{conj:bound} implies that 
\[
\max\{\lambda_2(G),\lambda_2(G^c)\} \geq 2-\sqrt2.
\]
\end{theorem}
\begin{proof}
We may assume $n\geq4$.  Conjecture 1 and Proposition \ref{prop:comp} imply that 
\begin{align*}
    \max\{\lambda_2(G),\lambda_2(G^c)\}&\geq\max\left\{\f_n(k),\f_n\left(\frac{n-2k}{2}\right)\right\}.
\end{align*}
As remarked previously, $\f_n$ is a decreasing function, and thus the lowest possible maximum will be achieved at the value of $k$ when these two terms equal.  This occurs for $k=\frac{n}{4}$.  Plugging this value of $k$ in, we find 
\begin{align*}
    \max\{\lambda_2(G),\lambda_2(G^c)\}&\geq\frac{\frac{3}{4}n+1-\sqrt{(\frac34n+1)^2-2n}}{2}.
\end{align*}
This expression is increasing in $n$, so it is at least its value at $n=4$.  Thus we obtain
\[
\max\{\lambda_2(G),\lambda_2(G^c)\}\geq  2-\sqrt{2}.
\]
\end{proof}

\subsection{Strengthening the Laplacian Spread Conjecture}\label{sec:strengthen}

While equality in the spread conjecture can be achieved, it is only achieved for graphs when $G$ or $G^c$ has a dominating vertex whose deletion yields a disconnected graph (see Theorem 1 of \cite{einollahzadeh2021proof}).  In particular, in any case of equality, either $G$ or $G^c$ has diameter 2. Moreover, it has been known for some time (see the proof of Theorem 2.4 in \cite{zhai2011laplacian}) that for all graphs of diameter less than or equal to 2, $\lambda_2(G)\geq1$, and thus the inequality (\ref{eq:lam2}) holds. It is of interest then to determine the minimum of $\lambda_2(G)+\lambda_2(G^c)$ over all graphs with diameter 3 and diameter 3 complement.

In Figure \ref{fig:curve}, we have plotted (for $n=7$) as a solid curve (in red) the ordered pairs $(f_n(k),f_n(\frac{n-2k}{2}))$ for $k\in (0,\frac{n}{2})$. 
Thus Conjecture 1 implies that the ordered pair $(\lambda_2(G),\lambda_2(G^c))$ always lies above the solid (red) curve. 
Further discussion of this figure will take place in Section \ref{sec:conjs}. 
In particular, when $G$ and $G^c$ have diameter 3, since the solid (red) curve is concave down (by the proof of Theorem \ref{thm:conj_impl_spread}), the worst case scenario is when $G$ has only 2 eccentricity 3 vertices, and $G^c$ has $n-2$, so we plug in $k=1$ and $k=n/2-1$ and we find that if $G$ and $G^c$ have diameter 3, then
\begin{align*}
\lambda_2(G)+\lambda_2(G^c)&\geq \frac{n-\sqrt{n^2-4n+8}}{2}+\frac{\frac{n}{2}+2-\sqrt{(\frac{n}{2}+2)^2-8}}{2}\\
&=\frac{2n-4}{n+\sqrt{n^2-4n+8}} + \frac{8}{n+4+\sqrt{(n+4)^2-32}}\\
&\geq \frac{2n-4}{2n}+\frac{4}{n+4}\\
&=1+\frac{2n-8}{n(n+4)}.
\end{align*}

Thus for diameter 3 graphs with $n>4$, our conjecture implies that $\lambda_2(G)+\lambda_2(G^c)$ is strictly larger than 1.

\begin{figure}
    \captionsetup{singlelinecheck=off, format=hang}
    \centering
    \includegraphics[width=4in]{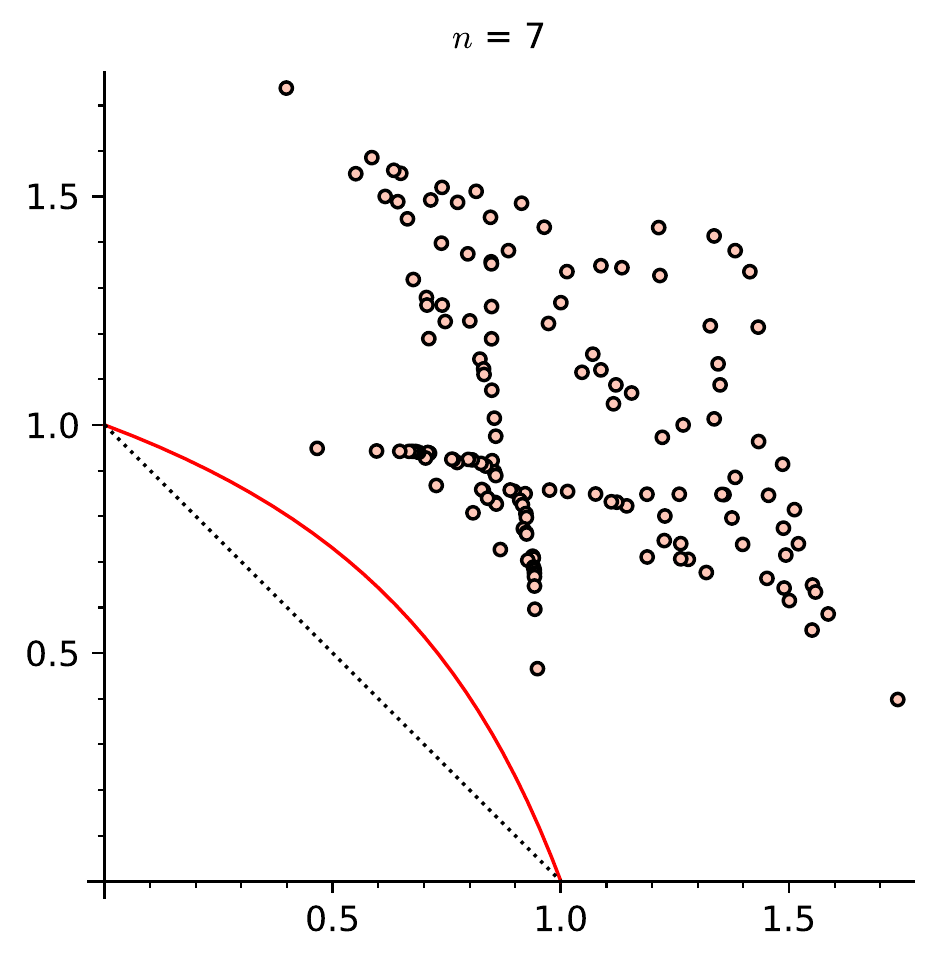}
    \caption[Points and conjectures for $n = 7$]%
    {On $n=7$ vertices there are $314$ graphs $G$ for which both $G$ and $G^c$ have diameter $3$, and for
    each of these the ordered pair $(\lambda_2(G),\lambda_2(G^c))$ is plotted.
    Also plotted are two conjectured lower bounds:
    
        \textbullet$\mspace{10mu}$Dotted (black), representing the line $\lambda_2(G)+\lambda_2(G^c) = 1$.
        
        \textbullet$\mspace{10mu}$Solid (red), containing all ordered pairs $(x, y)$ where $x$ is the
                                       bound from  Conjecture~1, and $y$ is the complementary bound
                                       (with $k$ replaced by $(n-2k)/2$).
    
     This plot will be repeated in Section~\ref{sec:conjs} with additional conjectures.
    }
    \label{fig:curve}
\end{figure}

 
\subsection{A Weaker Form of the Conjecture} \label{sec:weak}

As a side note, we observe that there is a weaker, but simpler form that Conjecture \ref{conj:bound} could take that still gives a relatively strong lower bound on $\lambda_2$.  

Consider the expression in (\ref{eq:conj_bound}) and observe that
\begin{align*}
    \frac{n-k+1-\sqrt{(n-k+1)^2-4(n-2k)}}{2} &= 
    \frac{2(n-2k)}{n-k+1+\sqrt{(n-k+1)^2-4(n-2k)}}\\
    &\geq\frac{2(n-2k)}{(n-k+1)+(n-k+1)}\\
    &=\frac{(n-k+1)-(k+1)}{n-k+1}\\
    &=1-\frac{k+1}{n-k+1}.
\end{align*}
Thus, a weaker version of Conjecture \ref{conj:bound} can be formulated as:

\begin{equation}\label{eq:weak_bound}
\lambda_2(G)\geq 1 - \frac{k+1}{n-k+1}.
\end{equation}

Once again, this expression is strictly decreasing in $k$.  Even though (\ref{eq:weak_bound}) is weaker than (\ref{eq:conj_bound}), it is simpler, and is a rational function, and thus may be easier to prove.

In Figure \ref{fig:curves_weak}, we plot this weaker bound in comparison to the stronger conjectured bound.

\begin{figure}
    \centering
    \includegraphics[width=2in]{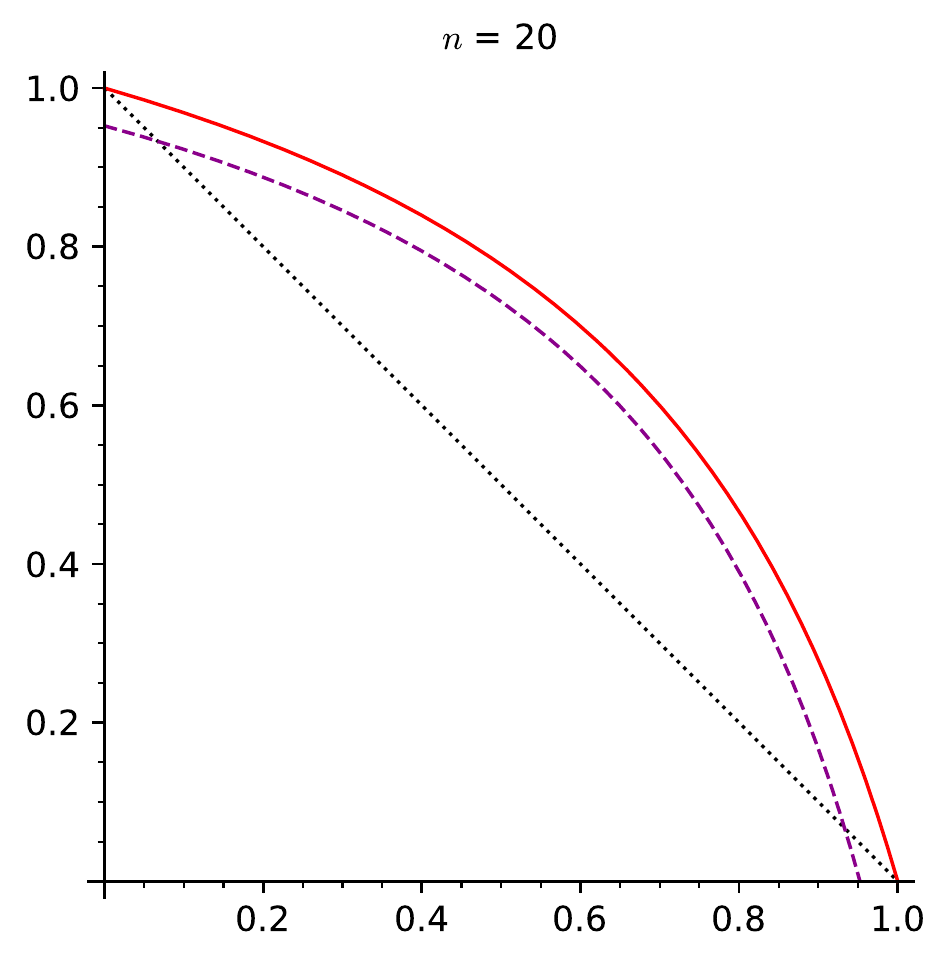}
    \caption{The dashed (purple) curve plots the weaker bound of Section \ref{sec:weak} compared to the solid (red) curve of the original bound, and the line $x+y=1$. }
    \label{fig:curves_weak}
\end{figure}

\subsection{Comparison with other known bounds}
We end this section by discussing comparisons between the bound of Conjecture \ref{conj:bound} and some other known lower bounds on $\lambda_2$.  Perhaps the most recognized lower bound on $\lambda_2$ is that of Mohar \cite{mohar1991eigenvalues} that involves the diameter $d$ of the graph:
\begin{equation}\label{eq:mohar}
    \lambda_2 \geq\frac{4}{nd}\ .
\end{equation}

Suppose $G$ is a diameter 3 graph with at least 1 vertex of eccentricity less than 3.  Then Conjecture \ref{conj:bound} gives a lower bound on $\lambda_2$ of at least $f_n(\frac{n-1}{2})$ while the bound from (\ref{eq:mohar}) gives a lower bound of $\frac{4}{3n}$.  A straightforward computation shows that $f_n(\frac{n-1}{2})>\frac{4}{3n}$ for any $n\geq5$.  Thus our bound is strictly stronger than the bound from (\ref{eq:mohar}) for any such graph.  Similarly, our bound is stronger for any diameter 4 graph with at least one vertex that has eccentricity less than 3.

An improvement on the bound in (\ref{eq:mohar}) was given by Lu et.~al.~\cite{lu2007lower} which states that 
\begin{equation}\label{eq:lu}
    \lambda_2\geq\frac{2n}{2+d(n(n-1)-2m)}
\end{equation}
where $d$ again represents the diameter, and $m$ is the number of edges in the graph.  Direct computation shows that for $d=3$, $f_n(k) \geq \frac{2n}{2+d(n(n-1)-2m)} $ if and only if 
\begin{align*}
36n^5&-72kn^4-144mn^3+288kmn^2-96n^4+168kn^3+84n^3+192mn^2-192kn^2\\&-24n^2+144m^2n-336kmn-48mn+112kn-288km^2+192km-32k \geq 0.
\end{align*}
Note that if $m=o(n^2)$ and $k=o(n)$, then for sufficiently large $n$, the bound of Conjecture \ref{conj:bound} is stronger than the bound of (\ref{eq:lu}).  In Figure \ref{fig:boundcompare}, for each diameter 3 graph on 7 vertices, we plot both the bound in Conjecture \ref{conj:bound} versus the true value of $\lambda_2$, and the bound from (\ref{eq:lu}) versus $\lambda_2$.  We can thus see empirically that the Conjecture \ref{conj:bound} bound is stronger in very many cases.

\begin{figure}[h!]
     \centering
     \begin{subfigure}[b]{0.49\textwidth}
         \centering
         \includegraphics[width=\textwidth]{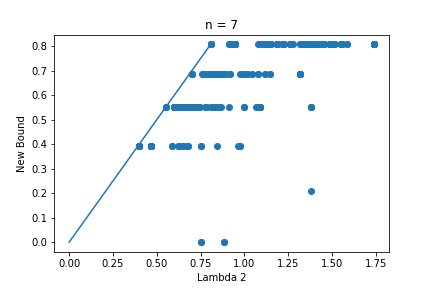}
         \label{fig:nb}
     \end{subfigure}
     \hfill
     \begin{subfigure}[b]{0.49\textwidth}
         \centering
         \includegraphics[width=\textwidth]{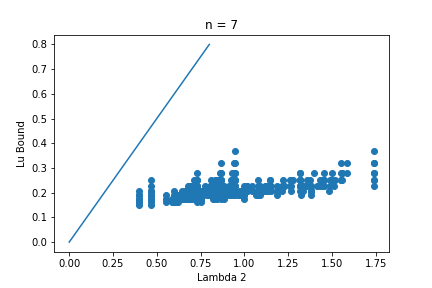}
         \label{fig:lu}
     \end{subfigure}
        \caption{A comparison of the bound in Equation~\eqref{conj:bound} and Equation~\eqref{eq:lu}.}
        \label{fig:boundcompare}
\end{figure}

\section{Some Families where the Conjecture is True}\label{sec:motivation}

In this section, we will show that Conjecture \ref{conj:bound} is true for certain families of graphs.  Close examination of one family in particular was the motivation to make the conjecture.  We will then closely examine a subclass of this family that we call ``dandelions" and will conjecture that the class of dandelions optimizes the spread over all diameter 3 graphs.

\subsection{Basic Observations}
There are some situations where the bound in Conjecture \ref{conj:bound} is trivial, or where it follows easily from known results.  We make note of those situations in this section.
Note that if $k=0$, then no vertices of $G$ have eccentricity 3 or more, so $G$ has diameter at most 2.  Plugging in $k=0$ into (\ref{eq:conj_bound}) yields the bound $\lambda_2(G)\geq1$.  As noted in the first paragraph of Section \ref{sec:strengthen}, this is known to be true for graphs of diameter 2 or less. 
Thus we have the following.
\begin{proposition}\label{prop:diam2}
Conjecture \ref{conj:bound} is true if the diameter of $G$ is at most 2.
\end{proposition}

Observe also that plugging in $k=n/2$ yields a bound of 0, which is trivial.  Observe that if the diameter of $G$ is 5 or more, then every vertex has eccentricity at least 3, so in that case $k=n/2$.  Thus we have the following.

\begin{proposition}\label{prop:diam5}
Conjecture \ref{conj:bound} is true if the diameter of $G$ is 5 or more.
\end{proposition}

Thus Conjecture \ref{conj:bound} only has significance for graphs of diameter 3 or 4.

We remark that any cycle of diameter 3 or more, because of the symmetry of the graph, will have all vertices of eccentricity 3 or more.  Thus by the same reasoning as above, we have the following fact.
\begin{proposition}
Conjecture \ref{conj:bound} is true for all cycles. 
\end{proposition}

\subsection{The Family $G(r,i,j)$}\label{sec:G}
 
 Let us denote by $G(r,i,j)$ the graph on $n=i+j+r+2$ vertices obtained as follows.  Begin with an edge between vertices $a$ and $b$.  Then attach $r$ vertices to both vertex $a$ and vertex $b$, then $i$ vertices to vertex $a$, and $j$ vertices to vertex $b$.  See Figure \ref{fig:grij}.

\begin{figure}
    \centering
    \begin{tikzpicture}
    \draw (-1,0)node{}--(1,0)node{}--(0,1)node{}--(-1,0) (1,0)--(0,2)node{}--(-1,0) (-1,0)--(-2,0)node{} (-1,0)--(-1.8,1)node{} (-1,0)--(-1.8,-1)node{} (1,0)--(1.8,.3)node{} (1,0)--(1.8,-.3)node{} (1,0)--(1.5,1)node{} (1,0)--(1.5,-1)node{};
    \end{tikzpicture}
    \caption{The graph $G(2,3,4)$}
    \label{fig:grij}
\end{figure}
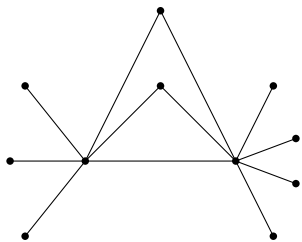
 
 Note that $|D(G(r,i,j))|=i+j$.  Empirical computation suggests that, among all diameter 3 graphs on $n$ vertices with $|D(G)|=\ell$, the graph with the smallest $\lambda_2$ is $G(n-\ell-2,\lfloor\ell/2\rfloor,\lceil\ell/2\rceil)$.  Thus the eigenvalues of this family of graphs are of interest.  
 
 We begin with a well-known lemma regarding ``twin vertices," that is, vertices with the same neighborhood.
 
 \begin{lemma}\label{lem:twins}
Let $G$ be a graph on $n$ vertices with $k$ non-adjacent twin vertices of common degree $d$.  Then $d$ is an eigenvalue of $L(G)$ of multiplicity at least $k-1$.  
\end{lemma}
\begin{proof}
Label the twin vertices vertices $1,2,...,k$ and their common neighbors $k+1,...,k+d$.  Let \[E_k=\{x\in\R^n:x_1+\cdots+x_k=0, x_{k+1}=\cdots=x_n=0\}.\]  We claim this is a $(k-1)$-dimensional eigenspace of $L(G)$ corresponding to eigenvalues $d$.  To see this, note $L(G)$ has the form
\[
L(G)=\begin{bmatrix}dI_k&-J_{k,d}&0\\-J_{k,d}&B&C\\0&C^T&D\end{bmatrix}
\]
where $J_{r,s}$ is the $r\times s$ matrix of all ones.  For $x\in E_k$ we can write $x=[y~0]^T$ where $\mathbf{1}_k^Ty=0$.  Then 
\[
L(G)x = \begin{bmatrix}dI_ky\\-J_{d,k}y\\0y\end{bmatrix}=dx.
\]
\end{proof}
 
 We will need the following Lemma, which is a standard result from the theory of equitable partitions (see \cite{cvetkovic}).  Recall that a partition of the vertex set $V(G)=V_1\dot\cup \cdots \dot\cup V_k$ is called an \emph{equitable partition} if for each $i,j\in\{1,...,k\}$ there are numbers $d_{ij}$ such that each vertex in $V_i$ has exactly $d_{ij}$ neighbors in $V_j$.
 
 \begin{lemma}[Lemma 7.1.6 of \cite{cvetkovic}]
Suppose $G$ has an equitable partition with parameters $d_{ij}$, $i,j=1,...,k$, and let $B=[b_{ij}]$ be the $k\times k$ matrix defined by \[b_{ij}=\begin{cases}-d_{ij}&i\neq j\\\sum_{s=1}^kd_{is}-d_{ij}&i=j\end{cases}.\]  Then any eigenvalue of $B$ is an eigenvalue of $L(G)$.  Moreover, any eigenvector of $L(G)$ coming from an eigenvalue of $B$ is constant on the parts of the partition.
\end{lemma}

 \begin{lemma}\label{lem:charpolyG_rij}
 The characteristic polynomial of the Laplacian matrix of $G(r,i,j)$ is 
 \begin{multline*}
 x(x-2)^{r-1}(x-1)^{i+j-2}(x^4-(2r+i+j+6)x^3+(r^2+ir+jr+8r+ij+4i+4j+13)x^2\\-(2r^2+2ir+2jr+10r+2ij+5i+5j+12)x+r^2+ir+jr+4r+2i+2j+4).
 \end{multline*}
 \end{lemma}
\begin{proof}

 Note that the graph $G(r,i,j)$ has an equitable partition with five parts: vertices $a$ and $b$ are each singleton vertices in their own part, the $r$ vertices attached to both consist of a part, the $i$ vertices attached to just $a$ are a part, and the $j$ vertices attached to just $b$ are a part. Thus $G(r,i,j)$ has equitable partition matrix
 \[
 Q=\begin{bmatrix}
 r+i+1&-1&-r&-i&0\\-1&r+j+1&-r&0&-j\\-1&-1&2&0&0\\-1&0&0&1&0\\0&-1&0&0&1
 \end{bmatrix}.
 \]
 Thus each eigenvalue of this matrix is an eigenvalue of the Laplacian of the graph.  The characteristic polynomial of $Q$ may be obtained by direct computation. This characteristic polynomial contains the factor $x$ and the quartic factor   \begin{multline*}
 x^4-(2r+i+j+6)x^3+(r^2+ir+jr+8r+ij+4i+4j+13)x^2\\-(2r^2+2ir+2jr+10r+2ij+5i+5j+12)x+(r^2+ir+jr+4r+2i+2j+4).
 \end{multline*}

 Observe that the $r$ vertices connected to both $a$ and $b$ are all twins, similarly the $i$ vertices connected to vertex $a$ are twins, and the $j$ vertices connected to vertex $b$ are twins.  Hence by Lemma \ref{lem:twins}, the spectrum of the Laplacian of the graph contains $r-1$ eigenvalues with value $2$ and $i+j-2$ eigenvalues with value $1$. Note that these eigenvalues coming from twins have eigenvectors that sum to 0 on the parts of the equitable partition, whereas the eigenvectors of the eigenvalues that come from $Q$ are constant on the parts, so these eigenvectors are orthogonal.  Thus, these eigenvalues coming from the twins are different from the eigenvalues that are roots of the characteristic polynomial of $Q$.  Hence, we have accounted for all the eigenvalues.
 \end{proof}

\begin{lemma}\label{lem:1root}
 The Laplacian of a graph with diameter 3 and diameter 3 complement has at most one eigenvalue in the interval $(0,1)$ and at most eigenvalue in the interval $(n - 1, n)$.
 \end{lemma}
 \begin{proof}
 Let $G$ be a graph with diameter 3 such that $G^c$ has diameter 3. By the symmetry of the spectrum
 with respect to complementation, it suffices to demonstrate that there is at most one eigenvalue
 in the interval $(0, 1)$.
 
 We first claim that $G$ must have at least one edge $e$ with the property that any vertex of $G$ is adjacent to at least one endpoint $e$.  We will refer to such an edge as a ``dominating edge."  To see that $G$ must have such an edge, let $a,b$ be vertices with $\dist_{G^c}(a,b)=3$.  Let $c$ be any other vertex of the graph.  Then $c$ cannot be adjacent to both $a$ and $b$ in $G^c$ since their distance is 3.  Thus $c$ is adjacent to either $a$ or $b$ in $G$.  There is an edge between $a$ and $b$ in $G$, so this is the desired edge.
 
 Now we claim that $G$ must contain as a spanning subgraph the tree $T=G(0,i,j)$ for some $i,j$ with $i+j=n-2$.  We can construct $T$ as follows.  Let $ab$ be a dominating edge in $G$ and put $ab$ into $E(T)$.  Then for every other vertex $c$ of $G$, add either the edge $ac$ or $bc$ to $T$.  Then $T$ will be a spanning tree of $G$ of the claimed type.  
 
 Now plugging in $r=0$ and $j=n-2-i$ into Lemma \ref{lem:charpolyG_rij} and simplifying, we find that the characteristic polynomial of $T$ is
 \[
 x(x-1)^{n-4}(x^3-(n+2)x^2+(ni-2i-i^2+2n+1)x-n).
 \]
 Let $p_3(x)$ be the cubic factor in this polynomial.  Observe that $p_3(0)=-n<0$, $p_3(1)=i(n-2-i)>0$,  $p_3(n-i)=-i<0$, and $p_3(n)=ni(n-2-i)>0$, so the three roots of $p_3(x)$ belong to the intervals $(0,1)$, $(1,n-i)$, and $(n-i,n)$.  
 
 So $L(T)$ has exactly 1 root in $(0,1)$. Then $G$ can be obtained from $T$ by adding edges, so $\lambda_3(L(G))\geq \lambda_3(L(T))\geq1$ (see Lemma 13.6.1 of \cite{godsilroyle}).  Thus $L(G)$ has at most 1 root in $(0,1)$.
 \end{proof}
 
\begin{theorem}\label{thm:Gfamily}
Conjecture 1 is true for the family $G(r,i,j)$.
\end{theorem}
\begin{proof}
First we claim that for a fixed $n$ and $r$, the smallest root of the characteristic polynomial is minimized when $i$ and $j$ are taken to be equal to each other.  To see this, we will reparameterize the coefficients of the characteristic polynomial by setting $i=q(n-r-2)$ and $j=(1-q)(n-r-2)$ and we will let $q$ vary between $0$ and $1$.  Doing this, the irreducible quartic in Lemma \ref{lem:charpolyG_rij} becomes
\begin{align*}
x^4-&(n+r+4)x^3+(q(q-1)(n-r)(4-n+r)+rn-4q(q-1)+2r+4n+5)x^2\\&-(2q(q-1)(n-r)(4-n+r)+2rn+r+5n-8q(q-1)+2)x+n(r+2).
\end{align*}
Evaluating this polynomial at $x=0$ yields $n(r+2)>0$ and at $x=1$ yields $q(q-1)(n-r-2)^2<0$ for $0<q<1$. So by Lemma \ref{lem:1root}, there is exactly 1 root strictly between 0 and 1, and this root occurs as the function goes from positive to negative.  

Differentiating the above polynomial with respect to $q$ gives
\[
(2q-1)(n-r-2)^2(2-x)x.
\]
Thus for $x\in(0,1)$, we see that there is a minimum when $q=1/2$.  So changing $q$ away from $1/2$ increases the values of the polynomial, so since the root occurred as the function was decreasing, the root moves to the right as we move $q$ away from $1/2$.  Thus the smallest possible root of this quartic is achieved when $q=1/2$.

Thus, going back to the polynomial in the expression from Lemma \ref{lem:charpolyG_rij}, the smallest root will be lower bounded by the root of the corresponding polynomial when we take $i=j=(n-r-2)/2$ (note that this need not be an integer).  Note that this value is $k=|D(G)|/2$.  Observe that when we replace $i=j=k$, then the quartic in the characteristic polynomial of $Q$ factors as
\[
\left(x^2-(r+k+3)x+(r+2)\right)\left(x^2-(r+k+3)x+(r+2k+2)\right)
\]
which, recalling that $r=n-2k-2$, can be rewritten as 
\[
\left(x^2-(n-k+1)x+(n-2k)\right)\left(x^2-(n-k+1)x+n\right).
\]
It is then straightforward to verify that the smallest nonzero root of this polynomial is the expression $f_n(k)$ in (\ref{eq:conj_bound}).
\end{proof}

We remark that, not only have we proven that Conjecture \ref{conj:bound} is true for $G(r,i,j)$, but we have also shown that when $i=j$, equality in Conjecture \ref{conj:bound} holds.  Indeed, it was analyzing this family of graphs that provided the original motivation for Conjecture \ref{conj:bound}.  Computation of many graphs suggests that the family $G(r,i,j)$ is the family of graphs that minimizes $\lambda_2$ among all $n$ vertex graphs with a fixed number of eccentricity 3 vertices.

\subsection{The Family $\widehat G(r,i,j)$}

For $r\geq 1$ we will denote by $\widehat G(r,i,j)$ the graph obtained from $G(r,i,j)$ by deleting the ``dominating edge," that is, the edge between vertices $a$ and $b$.  See Figure \ref{fig:Ghat}.  In this section we will prove that Conjecture \ref{conj:bound} is true for all $\widehat G(r,i,j)$.

\begin{figure}
    \centering
    \begin{tikzpicture}
    \draw (-1,0)node{} (1,0)node{}--(0,1)node{}--(-1,0) (1,0)--(0,-1)node{}--(-1,0) (-1,0)--(-2,0)node{} (-1,0)--(-1.8,1)node{} (-1,0)--(-1.8,-1)node{} (1,0)--(1.8,.3)node{} (1,0)--(1.8,-.3)node{} (1,0)--(1.5,1)node{} (1,0)--(1.5,-1)node{};
    \end{tikzpicture}
    \caption{The graph $\widehat G(2,3,4)$}
    \label{fig:Ghat}
\end{figure}
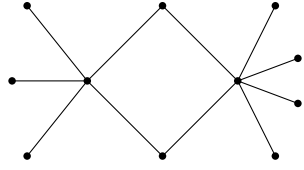

 \begin{lemma}\label{lem:charpolyGhat_rij}
 The characteristic polynomial of the Laplacian matrix of $\widehat G(r,i,j)$ is 
 \begin{multline*}
 x(x-2)^{r-1}(x-1)^{i+j-2}(x^4-(2r+i+j+4)x^3+(r^2+ir+jr+6r+ij+3i+3j+5)x^2\\-(2r^2+2ir+2jr+6r+2ij+2i+2j+2)x+r^2+ir+jr+2r).
 \end{multline*}
 \end{lemma}
 \begin{proof}

 Similar to the proof of Lemma \ref{lem:charpolyG_rij}, graph $\widehat G(r,i,j)$ has an equitable partition with five parts: vertices $a$ and $b$ are each singleton vertices in their own part, the $r$ vertices attached to both consist of a part, the $i$ vertices attached to just $a$ are a part, and the $j$ vertices attached to just $b$ are a part. The equitable partition matrix is
 \[
  Q=\begin{bmatrix}
 r+i&0&-r&-i&0\\0&r+j&-r&0&-j\\-1&-1&2&0&0\\-1&0&0&1&0\\0&-1&0&0&1
 \end{bmatrix}.
 \]
 Thus each eigenvalue of this matrix is an eigenvalue of the Laplacian of the graph.  The characteristic polynomial of $Q$ may be obtained by direct computation. This characteristic polynomial contains the factor $x$ and the quartic factor   \begin{multline*}
 x^4-(2r+i+j+4)x^3+(r^2+ir+jr+6r+ij+3i+3j+5)x^2\\-(2r^2+2ir+2jr+6r+2ij+2i+2j+2)x+r^2+ir+jr+2r.
 \end{multline*}
 
 The twin vertices work in exactly the same way as in the proof of Lemma \ref{lem:charpolyG_rij}, giving the remaining factors.
 \end{proof}
 
 \begin{theorem}\label{thm:Ghatfamily}
Conjecture 1 is true for the family $\widehat G(r,i,j)$.
\end{theorem}
\begin{proof}
We proceed as in the proof of Theorem \ref{thm:Gfamily}. We claim that for a fixed $n$ and $r$, the smallest root of the characteristic polynomial is minimized when $i$ and $j$ are taken to be equal to each other.  To see this, we will reparameterize the coefficients of the characteristic polynomial by setting $i=q(n-r-2)$ and $j=(1-q)(n-r-2)$ and we will let $q$ vary between $0$ and $1$.  Doing this, the irreducible quartic in Lemma \ref{lem:charpolyGhat_rij} becomes
\begin{align*}
x^4-&(n+r+2)x^3-(q(q-1)(n-r)^2+4q(q-1)(n+r-1)+r(n-1)-3n+1)x^2\\&-(2q(q-1)(n-r-2)^2-2n(r-1)+2)x+nr.
\end{align*}
Evaluating this polynomial at $x=0$ yields $nr>0$ and at $x=1$ yields $q(q-1)(n-r-2)^2<0$ for $0<q<1$. So there is a root between $0$ and $1$.  If there were 3 roots between $0$ and $1$, since the fourth root cannot be larger than $n$ (since it is an eigenvalue of a Laplacian on $n$ vertices), then since $r\geq1$, the sum of the eigenvalues of $Q$ could not reach $n+r+2$, which is the trace of $Q$.  Thus for all $q$, there is exactly one root between $0$ and $1$, and this root occurs as the function goes from positive to negative.  

Differentiating the above polynomial with respect to $q$ gives
\[
(2q-1)(n-r-2)^2(2-x)x,
\]
thus for $x\in(0,1)$, we see that there is a minimum when $q=1/2$.  So changing $q$ away from $1/2$ increases the values of the polynomial, so since the root occurred as the function was decreasing, the root moves to the right as we move $q$ away from $1/2$.  Thus the smallest possible root of this quartic is achieved when $q=1/2$.

Thus, going back to the polynomial in the expression from Lemma \ref{lem:charpolyGhat_rij}, the smallest root will be bounded below by the root of the corresponding polynomial when we take $i=j=(n-r-2)/2$ (note that this need not be an integer). Observe that when we replace $i=j=k$, then the quartic in the characteristic polynomial of $Q$ factors as
\[
\left(x^2-(r+k+1)x+r\right)\left(x^2-(r+k+3)x+(r+2k+2)\right)
\]
which, recalling that $r=n-2k-2$, can be rewritten as 
\[
\left(x^2-(n-k-1)x+(n-2k-2)\right)\left(x^2-(n-k+1)x+n\right).
\]
We thus see that $\lambda_2$ for this family of graphs is given by 
\[
\frac{n-k-1-\sqrt{(n-k-1)^2-4(n-2k-2)}}{2}.
\]

Note that all vertices of $\widehat G(r,i,j)$ have eccentricity 3 or more except the $r$ vertices connected to the two original vertices.  Thus in the context above, there are $2k+2$ vertices of eccentricity 3 or more.  Thus Conjecture \ref{conj:bound} is simply claiming that $f_n(k+1)\leq\frac{n-k-1-\sqrt{(n-k-1)^2-4(n-2k-2)}}{2}.$  A direct computation verifies that this is true.
\end{proof}

We remark that empirical computation on diameter $4$ graphs on small numbers of vertices suggests that the family $\widehat G(r,i,j)$ is the family of graphs that minimizes $\lambda_2$ over all diameter $4$ graphs with a fixed number of vertices at eccentricity $3$ or more.

\subsection{Trees}
In this section, we prove our bound holds for all trees.
\begin{theorem}
Conjecture \ref{conj:bound} is true for all trees.
\end{theorem}
\begin{proof}
By Propositions \ref{prop:diam2} and \ref{prop:diam5}, we need only consider trees of diameter 3 and 4. Observe that any tree of diameter 3 is simply $G(0,i,j)$ for some $i,j$, where the family $G(r,i,j)$ is defined previously.  By Theorem \ref{thm:Gfamily}, the conjecture holds for all of these.  In Theorem 3.2 of \cite{fallat1998extremizing}, it is proven that the diameter 4 trees on $n$ vertices that minimize $\lambda_2$ are exactly the trees $\widehat G(1,\lceil\frac{n-4}{2}\rceil,\lfloor\frac{n-4}{2}\rfloor)$ with $n$ vertices, where $\widehat G$ is the family defined in the previous section. Thus by Theorem \ref{thm:Ghatfamily} we are done.

\end{proof}

\subsection{Inserting edges}
\begin{lemma}[Lemma 13.6.1 of \cite{godsilroyle}]\label{lem:insert}
If $G'$ is obtained from $G$ by inserting an edge, then \[\lambda_2(L_{G'})\geq \lambda_2(L_G).\]
\end{lemma}

It follows that inserting an edge to a graph can only increase the algebraic connectivity.  Thus, if we know Conjecture \ref{conj:bound} is true for a graph $G$, then we automatically know it is true for any graph $G'$ obtained by inserting any edges that leave the number of vertices with eccentricity 3 or more unchanged.

For example, the graph in Figure \ref{fig:grijadded} was constructed by starting with $G(2,3,4)$ and inserting several edges that do not alter the eccentricity.  Thus we know that Conjecture \ref{conj:bound} is true for such a graph.  Indeed, in any $G(r,i,j)$, any edge within any of the parts of the equitable partition identified in the proof of Lemma \ref{lem:charpolyG_rij} do not alter any eccentricities, so inserting any number of these edges yields a graph for which Conjecture \ref{conj:bound} is true.  Several edges across these parts may also leave eccentricities unchanged.  Similar comments apply to the family $\widehat G(r,i,j)$.  Thus we have verified Conjecture \ref{conj:bound} for a very large collection of graphs.
\begin{figure}[h!]
    \centering
    \begin{tikzpicture}
    \draw (-1,0)node{}--(1,0)node{}--(0,1)node{}--(-1,0) (1,0)--(0,2)node{}--(-1,0) (-1,0)--(-2,0)node{} (-1,0)--(-1.8,1)node{} (-1,0)--(-1.8,-1)node{} (1,0)--(1.8,.3)node{} (1,0)--(1.8,-.3)node{} (1,0)--(1.5,1)node{} (1,0)--(1.5,-1)node{};
    \draw[color=red] (-2,0)--(-1.8,-1) (-2,0)--(-1.8,1) (-1.8,1)--(0,2) (0,1)--(0,2)
    (1.8,.3)--(0,1) (1.5,1)--(0,1) (1.8,.3)--(1.8,-.3)--(1.5,-1)--(1.8,.3);
    \end{tikzpicture}
    \caption{The graph $G(2,3,4)$ with several edges inserted.  No inserted edge has altered the eccentricity of any vertex.}
    \label{fig:grijadded}
\end{figure}
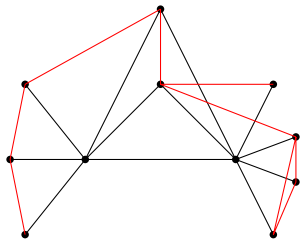

\section{Further remarks and conjectures}\label{sec:conjs}

We devote this final section to various conjectures based on empirical observations as
shown in Figure \ref{fig:curves}, which reproduces the unit square $[0, 1] \times [0, 1]$
from Figure \ref{fig:curve} with some additional curves.
From lowest to highest, the plots in Figure \ref{fig:curves} are:

        \textbullet$\mspace{10mu}$Dotted (black), representing the line $\lambda_2(G)+\lambda_2(G^c) = 1$.
        
        \textbullet$\mspace{10mu}$Dash-dot (blue), a conjectured lower bound for graphs weighted
                                       from the interval $[0, 1]$.
                                       
        \textbullet$\mspace{10mu}$Solid (red), containing all ordered pairs $(x, y)$ where $x$ is the
                                       bound from  Conjecture~1, and $y$ is the complementary bound
                                       (with $k$ replaced by $(n-2k)/2$).
                                       
        \textbullet$\mspace{10mu}$Dashed (green), an empirically observed algebraic lower bound.

\begin{figure}[H]
    \captionsetup{singlelinecheck=off, format=hang}
    \centering
    \includegraphics[width=2.48in]{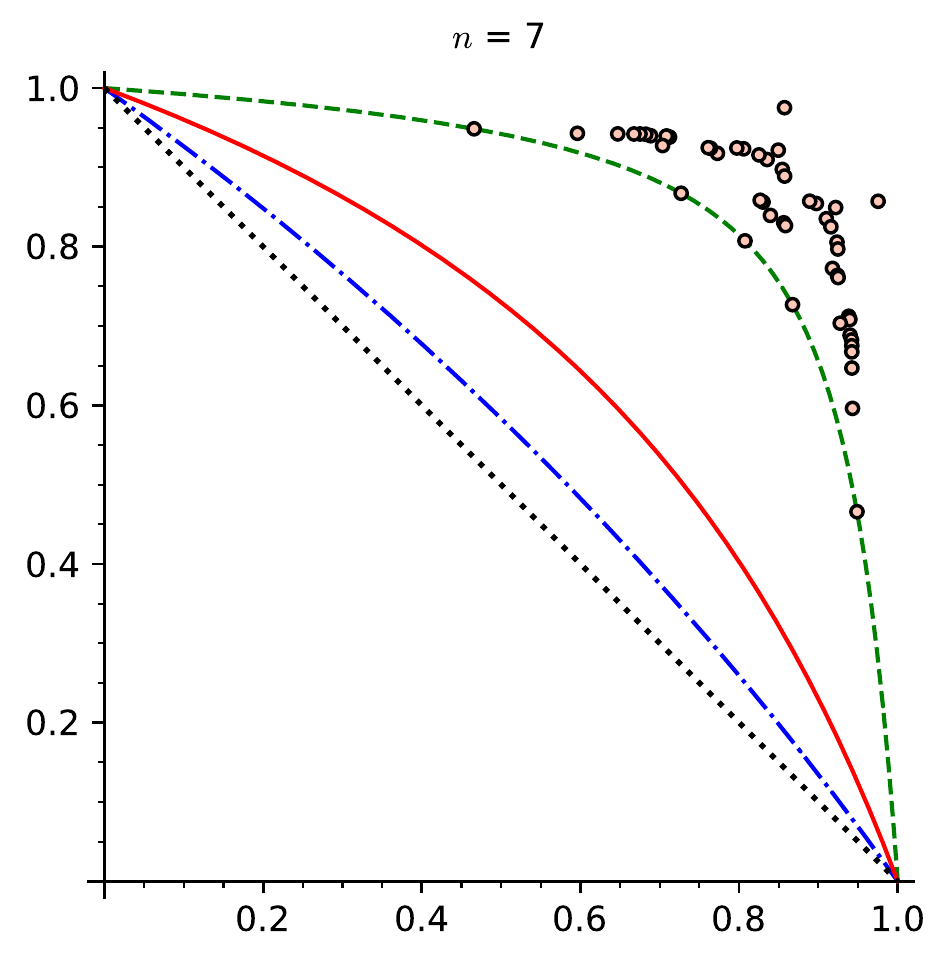}
    \caption[Points and four curves for $n = 7$]%
    {The same plot as Figure~\ref{fig:curve}, with two additional conjectured inequalities.
    }
    \label{fig:curves}
\end{figure}

\subsection{Dandelions}

\begin{definition}
The \emph{dandelion graph} on $n$ vertices is the graph $G(0,1,n-3)$ from Section \ref{sec:G}.  See Figure \ref{fig:dandelion}.
\end{definition}

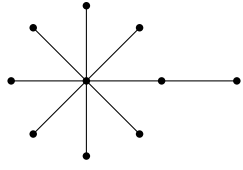
\begin{figure}
    \centering
    \begin{tikzpicture}
    \draw (-1,0)node{}--(0,0)node{}--(1,0)node{} (-1,0)--(-2,0)node{} (-1,0)--(-1.707,.707)node{} (-1,0)--(-1.707,-.707)node{} (-1,0)--(-1,1)node{} (-1,0)--(-1,-1)node{} (-1,0)--(-.2929,.707)node{} (-1,0)--(-.2929,-.707)node{} ;
    \end{tikzpicture}
    \caption{The dandelion on $10$ vertices.}
    \label{fig:dandelion}
\end{figure}

We have observed that among all graphs of diameter 3 and complement diameter 3, the dandelion on $n$ vertices and its complement represent the points on the dashed (green) curve of Figure \ref{fig:curves} closest to the corners.  We note that adding edges between the pendent twins leaves $\lambda_2$ unchanged.  Thus these graphs are of interest as for any fixed $n$, they appear to minimize $\lambda_2(G)+\lambda_2(G^c)$ over graphs of diameter 3.

\begin{conjecture}
The dandelion graph on $n$ vertices (and its complement) minimize $\lambda_2(G)+\lambda_2(G^C)$ over all graphs $G$ on $n$ vertices with diameter 3 and diameter 3 complement.
\end{conjecture}

In this section, we find nearly tight bounds on $\lambda_2(G)+\lambda_2(G^c)$ for these graphs. From Lemma \ref{lem:charpolyG_rij} we immediately have the following lemma.
\begin{lemma}\label{lem:dandpoly}
Let $G$ be the dandelion graph with $n$ vertices.  Then the characteristic polynomial of $G$ 
is $x(x-1)^{n-4}(x^3-(n+2)x^2+(3n-2)x -n)$.
\end{lemma}

\begin{lemma}\label{lem:4equations}
Let $p_n(x) = (x^3-(n+2)x^2 + (3n-2)x -n)$. Then for $n \geq 5$ 
\begin{equation}\label{eq:eq1}p_n(2-\phi) = -1,\end{equation} \begin{equation}\label{eq:eq2}p_n(2-\phi + 1/n) >0,\end{equation} \begin{equation}\label{eq:eq3}p_n(n-1) = -1,\end{equation} and \begin{equation}\label{eq:eq4}p_n(n-1+1/n) > 0,\end{equation}  where $\phi = \dfrac{\sqrt{5} + 1}{2}$.

\end{lemma}
\begin{proof}
We consider the four equations in turn.  \\
\noindent Equation~\eqref{eq:eq1}:
\begin{align*}p_n(2-\phi) & = 8 - 12\phi + 6 \phi^2 - \phi^3 - (n+2)(4-4\phi+\phi^2) + (3n-2)(2-\phi)-n\\
&=-\phi^3 + \phi^2(4-n) + \phi(n-2) + n-4\\
&= -(2\phi+1) + (\phi+1)(4-n) + \phi n - 2\phi+n-4\\
&= -1
\end{align*}
\noindent Equation~\eqref{eq:eq2}:
\begin{align*}p_n(2-\phi+1/n) & = -n\phi^2 + n\phi+n -\phi^3 + 4\phi^2 -5 + \frac{3\phi^2}{n}-\frac{8\phi}{n}+\frac{1}{n}-\frac{3\phi}{n^2}+\frac{4}{n^2} + \frac{1}{n^3}\\
&=\frac{1}{n^3}(4n - 3n\phi-5n^2 \phi + 2n^3\phi + 4n^2 - 2n^3 + 1).
\end{align*}
We note that if $n = 4$, this equation simplifies to $(36\phi -47)/64 \approx 0.1758>0$.  Since $1/n^3 > 0$ for $n>0$, to determine if $p_n(2-\phi+1/n) > 0$ we consider the term $g(n) = 4n - 3n\phi-5n^2 \phi + 2n^3\phi + 4n^2 - 2n^3 + 1$ and show that $g(n)> 0$ for all $n \geq 4.$ We note that $g'(n) = 4 - 3\phi-10n \phi + 6n^2\phi + 8n - 6n^2$, which is clearly greater than zero for $n = 4$, moreover $g''(n)  = n(12\phi-12) + 8-10\phi\approx 7.4164 n - 8.1803$ which is greater than zero for all $n \geq 2.$  Hence $g'(n)$ is an increasing function for $n \geq 2$, and since $g'(4) > 0$, we can say $g'(n) > 0$ for all $n \geq 0$ which implies that $g(n) > 0$ for $n \geq 4.$\\

\noindent Equation~\eqref{eq:eq3}:

\begin{align*}p_n(n-1) & = n^3-3n^2+3n-1-(n+2)(n^2-2n+1) + 3n^2-5n + 2 -n\\
&=-1.
\end{align*}
Finally, equation~\eqref{eq:eq4}:
\begin{align*}p_n(n-1+1/n) & = \frac{1}{n^3}(n^4-6n^3+7n^2-5n+1)\\
&=\frac{1}{n^3}(n^3(n-6)+n(7n-5)+1)
\end{align*}
Clearly this is positive when $n \geq 6$.  We observe that $p_4(4-1+1/4) \approx -0.5469$ and $p_5(5-1+1/5)\approx0.208.$

\end{proof}

\begin{theorem}
For the dandelion graph $G$ on $n$ vertices the algebraic connectivity of $G$ is in the interval $(2-\phi, 2-\phi+1/n)$, and $3-\phi-1/n<\lambda_2(G) + \lambda_2(G^C) < 3 -\phi+1/n$.
\end{theorem}
\begin{proof}
The characteristic polynomial of $G$ is given by $x(x-1)^{n-4}(x^3-(n+2)x^2+(3n-2)x -n)$, and the eigenvalues are $0$, $1$ and the roots of $p_n=x^3-(n+2)x^2+(3n-2)x -n$.  In particular $\lambda_2(G)$ is the smallest root of $p_n(x)$ and $\lambda_n(G)$ is the largest root of $p_n(x)$.  By Lemma~\ref{lem:4equations}, $2-\phi <\lambda_2(G) < 2-\phi+1/n$ and $n-1 < \lambda_n(G) < n-1+1/n$.  Hence $1-1/n < \lambda_2(G^C) < 1$ and $3-\phi-1/n<\lambda_2(G) + \lambda_2(G^C) < 3 -\phi+1/n$.
\end{proof}

We remark that $3-\phi\approx1.382$.

\subsection{An empirically observed bound on the symmetrized spread}

When plotting the pair $(\lambda_2(G), \lambda_2(G^c))$ for all graphs
on up to $10$ vertices, a pattern is observed in certain graphs
for which both values are simultaneously low, and in the values
that they achieve.
This pattern is observed for example in Figure~\ref{fig:curves} in which
certain points lie on the dashed (green) curve, but no points lie beneath it.

We describe three families of graphs on $n\ge4$ vertices,
each of which is defined with an associated pair of positive rational values $s,t$ with $s + t = 1$,
and state a conjectured bound for which exactly the graphs in these
families are supposed to be tight. Each family is modeled on a small
graph by replacing certain vertices by a cluster consisting of
one or more vertices.
Within a single cluster, vertex adjacencies are arbitrary.
If vertices $a$ and $b$ are adjacent and $a$ is replaced by
a cluster $A$, then every vertex in $A$ must be adjacent to $b$.
If vertices $a$ and $b$ are adjacent and both $a$ and $b$
are replaced by clusters $A$ and $B$, then every vertex in $A$
must be adjacent to every vertex in $B$.
The families are as follows:
\begin{enumerate}
    \item Thick-stemmed dandelions of the first kind. These
    are modeled on the path $(a, b, c, d)$ by replacing
    vertex $a$ with a cluster $A$ and replacing vertex $c$
    with a cluster $C$.
    The value of $s$ is $\frac{|C|}{n - 2}$ and the value
    of $t$ is $\frac{|A|}{n - 2}$.
    \item Thick-stemmed dandelions of the second kind.
    These are modeled on the path $(c, a, d, b)$ by replacing
    vertex $a$ with a cluster $A$ and replacing vertex $c$
    with a cluster $C$.
    The value of $s$ is $\frac{|A|}{n - 2}$ and the value
    of $t$ is $\frac{|C|}{n - 2}$.
    \item Generalized bull graphs. These are modeled on
    the bull graph with an induced path
    $(a, b, c, d)$ and a vertex $e$ adjacent to $b$ and $c$.
    The vertex $e$ is replaced by a cluster $E$.
    Every graph in this family
    has associated values $s = t = 1/2$.
\end{enumerate}
Figure \ref{fig:families_on_curve} gives an illustration of these three families of graphs. The complement of a thick-stemmed dandelion of the first kind
is a thick-stemmed dandelion of the second kind, and the
complement of a generalized bull graph is a generalized bull graph.

\begin{figure}
    \centering
    \begin{tikzpicture}
    \draw (-1,0)ellipse(10pt and 14pt) (0,0)node[label={[shift={(0,-.6)}]{$b$}}]{} (1,0)ellipse(10pt and 14pt) (2,0)node[label={[shift={(0,-.6)}]{$d$}}]{};
    \draw (-1,-.8)node[fill=white]{$A$}  (1,-.8)node[fill=white]{$C$};
    \draw (-1,.4)--(0,0) (-1,.2)--(0,0) (-1,0)--(0,0) (-1,-.2)--(0,0) (-1,-.4)--(0,0);
    \draw (.9,.4)--(0,0) (.9,.2)--(0,0) (.9,0)--(0,0) (.9,-.2)--(0,0) (.9,-.4)--(0,0);
    \draw (1.1,.4)--(2,0) (1.1,.2)--(2,0) (1.1,0)--(2,0) (1.1,-.2)--(2,0) (1.1,-.4)--(2,0);
    \end{tikzpicture}~~~~
    \begin{tikzpicture}
    \draw (-1,0)ellipse(10pt and 14pt) (1,0)node[label={[shift={(0,-.6)}]{$d$}}]{} (0,0)ellipse(10pt and 14pt) (2,0)node[label={[shift={(0,-.6)}]{$b$}}]{};
    \draw (0,-.8)node[fill=white]{$A$}  (-1,-.8)node[fill=white]{$C$};
    \draw (.1,.4)--(1,0) (.1,.2)--(1,0) (.1,0)--(1,0) (.1,-.2)--(1,0) (.1,-.4)--(1,0);
    \draw (-1,.4)--(-.1,.4)--(-1,.2)--(-.1,.2)--(-1,0)--(-.1,0)--(-1,-.2)--(-.1,-.2)--(-1,-.4)--(-.1,-.4)--(-1,-.2) (-.1,-.2)--(-1,0) (-.1,0)--(-1,.2) (-.1,.2)--(-1,.4);
    \draw (1,0)--(2,0);
    \end{tikzpicture}~~~~
    \begin{tikzpicture}
    \draw (-2,0)node[label={[shift={(0,-.6)}]{$a$}}]{}--(-1,0)node[label={[shift={(0,-.6)}]{$b$}}]{}--(1,0)node[label={[shift={(0,-.6)}]{$c$}}]{}--(2,0)node[label={[shift={(0,-.6)}]{$d$}}]{}(0,1)ellipse(10pt and 14pt);
\draw (0,.25)node[fill=white]{$E$}; 
\draw (-.1,1.4)--(-1,0) (-.1,1.2)--(-1,0) (-.1,1)--(-1,0) (-.1,.8)--(-1,0) (-.1,.6)--(-1,0);
    \draw (.1,1.4)--(1,0) (.1,1.2)--(1,0) (.1,1)--(1,0) (.1,.8)--(1,0) (.1,.6)--(1,0);
    \end{tikzpicture}
    \caption{Illustrations for thick-stemmed dandelions of the first and second kind, and the generalized bull family.}
    \label{fig:families_on_curve}
\end{figure}
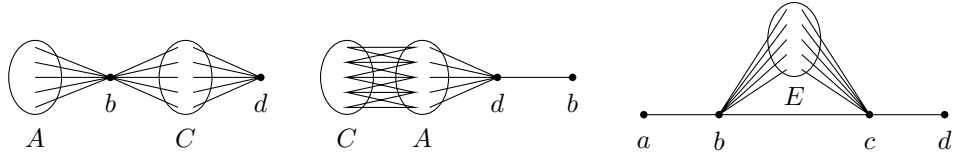

\begin{theorem}
\label{thm:greenpoints}
Let $G$ be a graph on $n > 4$ vertices which is a thick-stemmed dandelion of the first kind,
a thick-stemmed dandelion of the second kind, or a generalized bull graph.
Let $x$ denote $\lambda_2(G)$ and let $y$ denote $\lambda_2(G^c)$.
Then $x$ and $y$ satisfy the equation
\[
xy(2-xy)
=
n(1-x)(1-y)(n-2-x-y).
\]
\end{theorem}

\begin{proof}

Note that the thick-stemmed dandelions of the first kind have an equitable partition with quotient matrix
\[
\begin{bmatrix}
1&-1&0&0\\-(1-s)(n-2) & n-2 &-s(n-2)&0\\0&-1&2&-1\\0&0&-s(n-2)&s(n-2)
\end{bmatrix}
\]
whose characteristic polynomial is $xw_n(x,s)$ where

\[
w_n(x, s) = x^3 + (-sn - n + 2s - 1)x^2 + (n^2s - 2ns + 2n - 2)x - n^2s + 2ns.
\]

Observe that
\begin{equation*}
\begin{aligned}
w_n(0)&= ns(2-n) < 0\\   w_n(1)&= (n-2)(1-s) > 0\\ w_n(n-1) &= -(n-2)s < 0,\\ w_n(n) &= n(n-2)(1-s) > 0.
\end{aligned}
\end{equation*}\medskip

So  there is exactly one root of $w_n$ in $(0,1)$ and exactly one in $(n-1,n)$.  Then by Lemma \ref{lem:1root}, these are the smallest and largest roots of the characteristic polynomial for the Laplacian of the thick-stemmed dandelion.

Observe that this polynomial has the symmetry
\[
w_n(n - x, 1 - s) = -w_n(x, s)
\]
and thus it is true both that $x = \lambda_2(G)$ is a root of $w_n(x, s)$
and that $y = \lambda_2(G^C) = n - \lambda_n(G)$ is a root of $w_n(y, t)$.
Since the desired identity is symmetric in $x$ and $y$,
in order to prove the desired identity for thick-stemmed dandelions of the second
kind it will suffice to prove it for their complements the thick-stemmed dandelions
of the first kind.

For generalized bull graphs, there is an equitable partition with quotient matrix
\[\begin{bmatrix}
1&-1&0&0&0\\-1&n-2&-(n-4)&-1&0\\0&-1&2&-1&0\\0&-1&-(n-4)&n-2&-1\\0&0&0&-1&1
\end{bmatrix}\]
which has characteristic polynomial
\[
x(x^2-nx+n)(x^2-nx+n-2).
\]
Note that the smallest root is the smallest root of $x^2-nx+n-2$, which is a factor of $w_n(x,1/2)$.

Having established that in all three families
$x$ and $y$ satisfy the identities $w_n(x, s) = 0$ and $w_n(y, t) = 0$
for the appropriate values of $s$ and $t$, it remains to show that $x$ and $y$ together
satisfy the identity stated in Theorem~\ref{thm:greenpoints}.
A certain rational function will prove useful in this aim.
Since $w_n(x, s)$ is linear in $s$, it is easy enough to solve the identity $w_n(x, s)=0$ for $s$
and obtain a rational function that we name $r_n(x)$:
\begin{eqnarray*}
r_n(x) &:=& \frac{x^3 - nx^2 - x^2 + 2nx - 2x}{nx^2 - 2x^2 - n^2x + 2nx + n^2 - 2n} \\
&=& \frac{x(x-2)(x-n+1)}{ x(x-2)(x-n+1) \ -\  (x-1)(x-n+2)(x-n)} \\
&~& \\
w_n(x, s) = 0 &\Longrightarrow& s = r_n(x) \\
w_n(y, t) = 0 &\Longrightarrow& t = r_n(y).
\end{eqnarray*}

Starting then from the identity
$$
s + t = 1
$$
we obtain
\[
r_n(x) + r_n(y) - 1 = 0
\]
whose numerator, after canceling and factoring, yields the polynomial identity
\[
(-x - y + n)(x^2y^2 - nx^2y - nxy^2 + n^2xy + nx^2 + ny^2 - n^2x - n^2y - 2xy + nx + ny + n^2 - 2n) = 0.
\]
Given $x < 1$, $y < 1$, and $n > 2$, the first factor is nonzero, yielding, after
regrouping and factoring, the desired identity
\[
xy(2-xy)
=
n(1-x)(1-y)(n-2-x-y).
\]
\end{proof}

\begin{corollary}\label{cor:max}
Let $G$ on $n$ vertices be a generalized bull graph, or, for $n$ even,
a thick-stemmed dandelion of either the first or second kind with $|A| = |C|$.
Then $G$ achieves
\[
\frac{n-3}{n-2} <
\max\{\lambda_2(G),\lambda_2(G^c)\} =
\frac{n - \sqrt{(n - 2)^2 + 4}}{2}
< \frac{n-2}{n-1}.
\]
\end{corollary}

\begin{proof}
In fact $\lambda_2(G)$ and $\lambda_2(G^c) = n - \lambda_n(G)$ are equal,
coming from $s = t = 1/2$ and the symmetric roots of the factor $x^2 - nx + n - 2$
of $w_n(x, 1/2)$.
\end{proof}

\begin{conjecture}\label{conj:emp}
Let $G$ be a graph on $n$ vertices such that $x = \lambda_2(G)$ and $y = \lambda_2(G^c)$
are both strictly less than $1$.
Then $x$ and $y$ satisfy the inequality
\[
xy(2-xy)
\ge
n(1-x)(1-y)(n-2-x-y)
,
\]
with equality only in the case that $G$ belongs to one of the three families
enumerated above.
\end{conjecture}

Observe that for any given pair $(x, y)$ in the range $0 < x, y < 1$ that
satisfies the conjectured inequality, all entrywise greater pairs
$x^\prime \ge x$, $y^\prime \ge y$ in the allowed range also satisfy
the inequality, because within this range
the right side of the inequality is decreasing in both
$x$ and $y$
and the left side of the inequality is increasing in both $x$ and $y$,
as can be seen for $x$ by
\[
\frac{\partial}{\partial x}(xy(2 - xy)) = 2y(1-xy)
\]
and for $y$ symmetrically.

\begin{figure}
    \captionsetup{singlelinecheck=off, format=hang}
    \centering
    \includegraphics[width=3.6in]{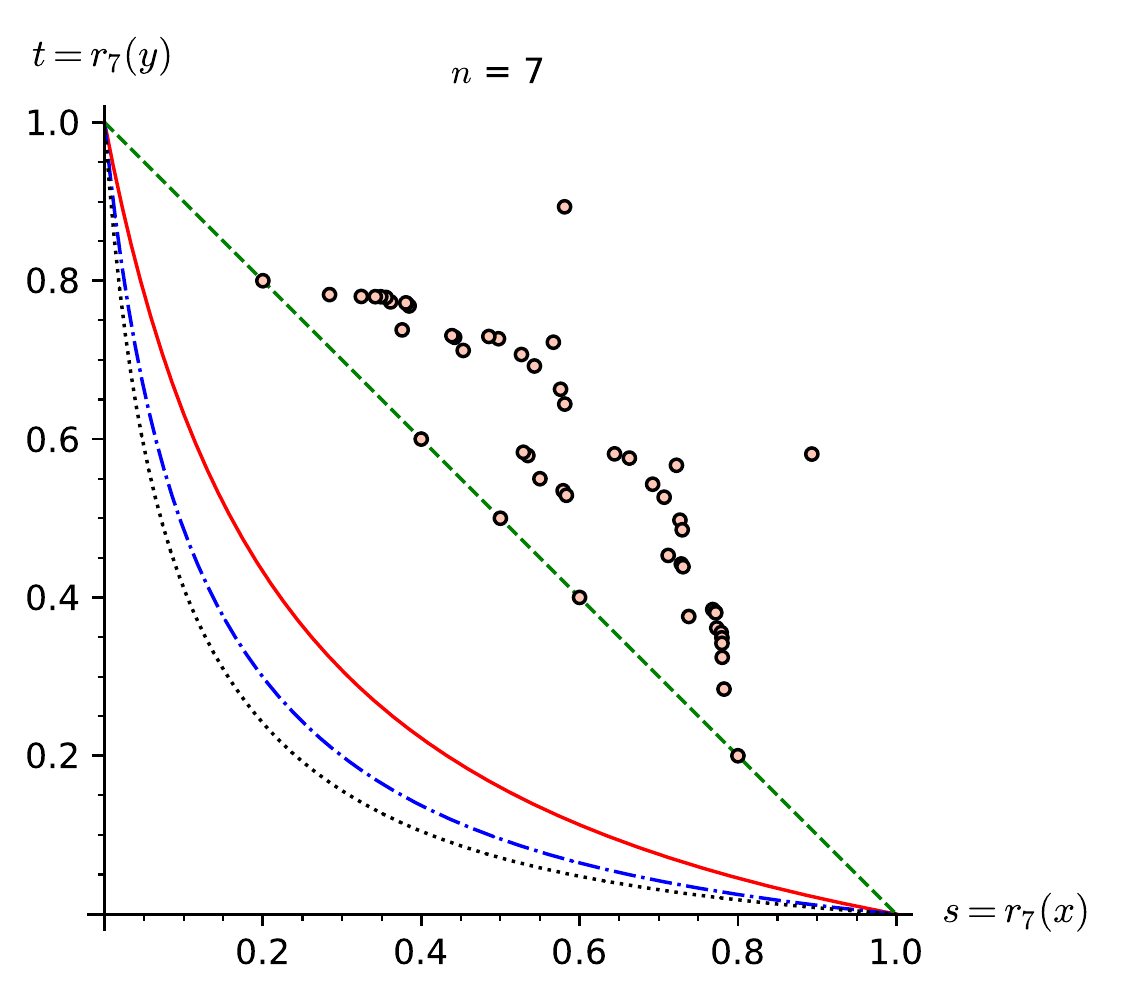}
    \caption[Shifted points and four curves for $n = 7$]%
    {The same curves and points as 
    in Figure~\ref{fig:curves},
    reparameterized by the rational function $r_7$ to display pairs
    $(s, t) = (r_7(x), r_7(y))$.}
    \label{fig:shifted_curves}
\end{figure}

Some remarks about the rational function $r_n(x)$: It has two vertical asymptotes coming from the real roots of its
quadratic denominator,
rotational symmetry around the point $(n/2, 1/2)$, and
a slant asymptote of $y - 1/2 = (x - n/2)/(n - 2)$.
Each of its three branches is strictly increasing with range $(-\infty, \infty).$
It passes through the six points $(0, 0), (1, 1), (2, 0), (n - 2, 1), (n - 1, 0),$ and $(n, 1).$
It is increasing and concave upward from the point $(0, 0)$ to the point $(1, 1)$ and thus
represents an invertible rescaling of the interval $[0, 1]$ that shifts all interior points leftwards.
Indeed, as illustrated in Figure~\ref{fig:shifted_curves},
it represents precisely the rescaling which, if applied to both axes of the unit square $[0, 1] \times [0, 1]$,
shifts the dotted (green) curve of Figure~\ref{fig:curves} to a straight
line joining the points $(0, 1)$ and $(1, 0)$.  Thus Conjecture \ref{conj:emp} implies the symmetric formulation of the Laplacian Spread Conjecture.
%

Recall in the definitions of these families of graphs, $s$ is given as a parameter corresponding to the graph: $s=\frac{|C|}{n-2}$ for the thick-stemmed dandelions of the first kind.  The points along the straight dotted (green) line $s+t=1$ of Figure~\ref{fig:shifted_curves}
that correspond to the thick-stemmed dandelion families are equally spaced along the line
under the reparameterization.
For instance, for the case $n=7$ (as in the figure) these families lie on the points $(\frac15,\frac45)$, $(\frac25,\frac35)$, $(\frac35,\frac25)$, $(\frac45,\frac15)$. Every graph in the generalized bull family corresponds to the
midpoint $(1/2, 1/2)$ of the reparameterized line.

Corollary \ref{cor:max} and Conjecture \ref{conj:emp} give us a natural conjecture on the lower bound for the maximum of $\lambda_2(G)$ and $\lambda_2(G^c)$ that strengthens the $1-O(n^{-1/3})$ bound found in \cite{einollahzadeh2021proof}.

\begin{conjecture}
For any graph $G$ on $n$ vertices,
\[
\max\{\lambda_2(G),\lambda_2(G^c)\} \ge
\frac{n - \sqrt{(n - 2)^2 + 4}}{2}
> \frac{n-3}{n-2}.
\]
\end{conjecture}







\subsection{Conjectures on weighted graphs}
For a given labeling of $n$ graph vertices, the set of Laplacians of
simple graphs are the geometric extreme points
of a cube of dimension $\binom{n}{2}$ each point of which is
the Laplacian of a weighted graph $G$ in which all weights lie in the interval
$[0, 1]$.
For any such weighted graph $G$ the algebraic connectivity
$\lambda_2(G)$ and spectral radius $\lambda_n(G)$ can also be defined,
and it is a natural question whether the Spread Conjecture also holds
for such weighted graphs.
Investigations on families of such graphs and on small examples of random
graphs appear so far to satisfy the Spread Conjecture.

There is a natural operation of complementation on weighted graphs
that is obtained by replacing every edge weight $w_{ij}$ by
$w_{ij}^\prime = 1 - w_{ij}$, and the same identities hold
for the Laplacian eigenvalues of the complement of a weighted
graph as hold under complementation of simple graphs.
In particular, we have
\[
\lambda_2(G^c) = n - \lambda_n(G).
\]
The natural extension of the symmetric formulation of the Spread Conjecture to weighted graphs
can thus be stated as follows:

\begin{conjecture}
Let $G$ be a weighted graph with all weights in the range $[0, 1]$.
Then
\[
\lambda_2(G) + \lambda_2(G^c) \ge 1.
\]
\label{conj:weighted_linear}
\end{conjecture}

Two other conjectured bounds have been introduced. The symmetric
formulation of Conjecture \ref{conj:emp} is represented
by the dashed (green) curve, and the symmetric formulation of Conjecture \ref{conj:bound} is represented by the solid (red) curve of Figure~\ref{fig:curves}.
It is natural to ask whether these inequalities might
be satisfied also for weighted graphs, but the answer
is that, for $n \ge 5$ vertices, there are weighted graphs
that fail these inequalities.
Investigations have identified a family of weighted graphs that appear to be spectrally
extreme with respect to the Laplacian spread.
The graphs in this family depend on a parameter $s$,
and the values of $\lambda_2(G)$ and $\lambda_2(G^c)$
that are achieved as $s$ varies for this family
delineate what appears, in our investigations to date,
to be the strongest possible inequality
over weighted graphs.

\begin{definition}
  An \emph{SE$(s)$ graph} is a weighted graph $G$ with all edge weights
  in the interval $[0, 1]$ and with vertices $u$ and $v$ such that:
\begin{itemize}
    \item the edge weight between $u$ and any vertex in $V(G) \setminus \{u, v\}$ is $0$,
    \item the edge weight between $v$ and any
    vertex in $V(G) \setminus \{u, v\}$ is $1$, and
    \item the edge weight between $u$ and
    $v$ is $s \in (0, 1)$.
\end{itemize}
\end{definition}

Observe that the complement of an SE$(s)$ graph is an SE$(t)$
graph, for $s + t = 1$, with the roles of $u$ and $v$ exchanged.
The SE$(s)$ graphs are, in retrospect, a natural family of weighted graphs
to consider when trying to disconnect, as much as possible, both
the weighted graph and its complement.
A simple graph $G$ is most easily disconnected by having any one vertex of degree~$0$;
symmetrically, it is easiest to disconnect $G^c$
by having any one vertex with degree $n - 1$ in $G$.
A graph and its complement cannot be simultaneously disconnected,
of course, but they could be
if it were possible to have both a vertex $u$ of degree~$0$ and a vertex~$v$ of
degree $n - 1$.
The impediment to simultaneous disconnection comes from
the fact that $u$ and $v$ cannot agree on the number
of edges between them, a number that cannot be simultaneously
$0$ and $1$.
In the class of weighted graphs, $u$ and $v$ can compromise on this number
using any edge weight $w_{uv} = s$ in the range $[0, 1]$
(and edge weight $t = 1 - s$ in the complement),
giving a one-parameter family of graphs that are poorly
connected simultaneously with their complements.
In fact we will show the full $\mathrm{SE}(s)$ family consists of a great many such
one-parameter families, because the non-parametrized weight
of any edge involving neither $u$ nor $v$ is arbitrary
within its legal range of $[0, 1]$.

We first establish a relationship between the extremal
Laplacian eigenvalues, or equivalently between the algebraic
connectivity and complementary algebraic connectivity,
of weighted graphs within the SE$(s)$ family.

\begin{theorem}
Let $G$ be an SE$(s)$ graph on $n$ vertices,
with algebraic connectivity $x = \lambda_2(G)$ and complementary algebraic connectivity
$y = \lambda_2(G^c)$.
Then $x$ and $y$ satisfy the equation
\[
x + y -\frac{2xy}{n} = 1,
\]
and the positive values $s$ and $t = 1 - s$ satisfy the equations
\[
s = \frac{x - y + 1}{2}, \ \ t = \frac{y - x + 1}{2}.
\]
\end{theorem}

\noindent\textbf{Remark.} This curve is the hyperbola centered at $(\frac{n}{2},\frac{n}{2})$ given by $\left(x-\frac{n}{2}\right)\left(y-\frac{n}{2}\right)=\frac{n}{4}(n-2)$ and represents the dotted dashed (blue) curve of Figure \ref{fig:curves}.  The lower left branch of the hyperbola approaches the line $x+y=1$ as $n\rightarrow \infty$.

\begin{proof}
We calculate $x$ and $y$ in the cases that 

\begin{itemize}
\item every edge not incident to $u$ nor $v$ has weight $0$ and
\item every edge not incident to $u$ nor $v$ has weight $1$.
\end{itemize}  

We show that $x$ and $y$ have the same value in either case, from which it follows that $x$ and $y$ have the same values for all other allowed edge weights.  

First we consider the case in which every edge not incident to $u$ nor $v$ has weight $0$. Call this graph $G_0(s)$.  Then taking $u$ to be vertex 1 and v to be vertex 2, its Laplacian matrix is 

\[
L_0= \begin{bmatrix}
s & -s & 0 & 0 & \dots & 0\\  -s & s+n-2 & -1 & -1 & \dots & -1\\ 0 & -1 & 1 & 0 & \dots & 0\\ 0 & -1 & 0 & 1 &  & 0\\
\vdots & \vdots & \vdots & \vdots & \ddots & 0\\ 0 & -1 & 0 & 0 & & 1
\end{bmatrix}
\]

In order to calculate the characteristic polynomial we introduce

\[ 
S=\begin{bmatrix} 1 & 0 & 0 & 0 & \cdots &0\\ 0& 1 & 0 & 0 & \cdots & 0 \\ 0 & 0 & 1 & 0 & \cdots & 0\\ 0 & 0 & -1 & 1 &  & 0 \\ 0 & 0 & \vdots & 0 & \ddots  & 0 \\ 0 & 0 & -1 & 0 & & 1
\end{bmatrix} 
\]

A calculation gives 

\[
S (L_0 - I) S^{-1}= \begin{bmatrix} s-1 & -s & 0 & 0 & \cdots & 0\\ -s & s+n-3 & 2 - n & -1 & \cdots & -1 \\ 0 & -1 & 0 & 0 & \cdots & 0 \\ 0 & 0 & 0 & 0 & \cdots & 0 \\ \vdots & \vdots & \vdots & \vdots & \cdots & \vdots \\  0 & 0 & 0 & 0 & & 0 \end{bmatrix}.
\]

A straightforward calculation shows that the characteristic polynomial, in $w$, of $L_0-I$ is 

\[ w^{n-3} [w^3 - (2s+n-4)w^2 + ((n-4)s+5-2n)w +(n-2)s-n+2].\]

Replacing $w$ by $z-1$ gives the characteristic polynomial of $L_0$ in $z$:

\[
 z(z-1)^{n-3}[z^2 -(2s+n-1)z + ns].
\]
Let $p(z)$ be the quadratic factor,
and evaluate it at the following four points:
\begin{itemize}
    \item $p(0) = ns > 0$,
    \item $p(1)= (n-2)(s-1) < 0$,
    \item $p(n-1) =-s(n-2) < 0$, and
    \item $p(n) = n(1-s) > 0$.
\end{itemize}

It follows that there is a root of $p$ between $0$ and $1$ and
another root of $p$ between $n-1$ and $n$,
which must be respectively $\lambda_2(G_0(s))$ and $\lambda_n(G_0(s))$.
Then
\begin{eqnarray*}
x+y-\frac{2xy}n &=& \lambda_2  + n - \lambda_n -(2/n) \lambda_2 (n-\lambda_n) \\
&=& n - (\lambda_2 + \lambda_n) +(2/n) \lambda_2 \lambda_n \\
&=& n - (2s+ n-1) +(2/n) ns \\
&=& n-2s - n +1 +2s \\
&=&1,
\end{eqnarray*}
which establishes the first claim for $G_0(s)$.
For the second claim, we have
\begin{eqnarray*}
\frac{x - y + 1}{2} &=& (1/2)(\lambda_2 - (n - \lambda_n) + 1) \\
&=&(1/2)(1 - n + (\lambda_2 + \lambda_n)) \\
&=&(1/2)(1 - n + (2s + n - 1)) \\
&=& s
\end{eqnarray*}
and
\[\frac{y - x + 1}{2} \ = \ 1 - \frac{x - y + 1}{2}
\ =\  1 - s
\ =\  t.
\]
It remains to establish the same claims for the graph $G_1(s)$
in which all weights involving neither $u$ nor $v$ are equal to $1$.
Observe however that the complement of $G_1(s)$ is
$G_0(t)$, and that the desired claims are symmetric
under the exchange of both the pair $(s, t)$ and the pair $(x, y)$.
\end{proof}


We conjecture that the name SE$(s)$ can be read as
\emph{spectrally extreme} or \emph{spread extreme},
namely that
these graphs give the extreme values
of the Laplacian spread for weighted graphs with non-negative weights
whose complements also have non-negative weights.

\begin{conjecture}\label{conj:weighted}
Let $G$ be a weighted graph on $n \ge 5$ vertices,
all of whose edge weights lie in the interval $[0, 1]$,
with $x = \lambda_2(G)$ and $y = \lambda_2(G^c)$.
Then $x$ and $y$ satisfy the inequality
\[
x + y -\frac{2xy}{n} \ge 1,
\]
with equality in the case that $G$ is an SE$(s)$ graph.
%
%
%
\label{conj:weighted}
\end{conjecture}

\noindent\textbf{Remark.} Our observations lead us to believe further that if we impose the strict inequality $\lambda_2(G)<1$ and $\lambda_2(G^c)<1$, then these are the only graphs achieving equality.  Strict inequality is necessary since we know there are graphs not of these families that lie on the corners $(0,1)$ and $(1,0)$.

Observe that Conjecture~\ref{conj:weighted} strictly implies
Conjecture~\ref{conj:weighted_linear}, and that
for any given pair $(x, y)$ in the range $0 < x, y < 1$ that
satisfies Conjecture~\ref{conj:weighted}, all entrywise greater pairs
$x^\prime \ge x$, $y^\prime \ge y$ in the allowed range also satisfy
the inequality, because within this range
the left side of the inequality is increasing in both $x$ and $y$,
as can be seen for $x$ by
\[
\frac{\partial}{\partial x}\left( x + y -\frac{2xy}{n}\right) = 1 - (2/n)y
\]
and for $y$ symmetrically.\\


\noindent\textbf{Remark.} We have observed that for $n>4$, the dashed (green) curve lies strictly above the solid (red) curve, which in turn lies strictly above the dashed dotted (blue) curve in Figure \ref{fig:curves}.  Furthermore, we remark that all the curves  coincide when $n=4$, simplifying to the hyperbola $(x-2)(y-2)=2$.    For $n=2$ and $n=3$ all possible connected graphs have diameter 2 or less, thus any of these graphs will lie on or above the corners in Figure \ref{fig:curves}.

\section{Appendix}

\begin{lemma}\label{lem:decr}
The expression in (\ref{eq:conj_bound}) is decreasing in $k$.
\end{lemma}
\begin{proof}
Define $\f_n(k) = \frac12\left(n-k+1 -\sqrt{(n-k+1)^2-4(n-2k)}\right)$.  We will prove that $\f_n'(k)$ is negative.
For $n\geq2$ we have 
\begin{align*}
    \f_n'(k)&=\frac12\left(-1 - \frac{2(n-k+1)(-1)+8}{2\sqrt{(n-k+1)^2-4(n-2k)}}\right)\\
    &=\frac12\left(-1+\frac{n-k-3}{\sqrt{n^2-2nk-2n+6k+k^2+1}}\right)\\
    &=\frac12\left(-1+\frac{n-k-3}{\sqrt{n^2-2nk-6n+6k+k^2+9 +(4n-8)}}\right)\\
    &=\frac12\left(-1+\frac{n-k-3}{\sqrt{(n-k-3)^2+(4n-8)}}\right)\\
    &<\frac12\left( -1 + \frac{n-k-3}{n-k-3}\right)=0.
\end{align*}
\end{proof}
\begin{lemma}\label{lem:concavedown}
The expression for $g_n(k)$ in \eqref{eq:gnkcd} is concave down on the interval $k \in\left(0,\frac{n}{2}\right)$.
\end{lemma}
\begin{proof}
We first observe that 
\[g_n(k) = \frac{1}{2}\left(\frac{3n}{2} - \sqrt{(n-k-1)^2+4k} - \sqrt{(k + n/2 + 1)^2 - 8k}\right).\]
Moreover, 
\[ g_n'(k) = -\frac{2k + n - 6}{4\sqrt{(k + n/2 + 1)^2 - 8k}} - \frac{2k - 2n + 6}{4\sqrt{(n - k - 1)^2} +4k},\]
and 
\[ g_n''(k) =-\frac{(2n-4)(((n-k-1)^2+4k)^{3/2} + ((k + n/2 + 1)^2 - 8k)^{3/2})}{((n-k-1)^2+4k)^{3/2}((k + n/2 + 1)^2 - 8k)^{3/2}}.\]
By prior arguments it is straightforward to see that for $k\in (0,n/2)$, $(k + n/2 + 1)^2 - 8k> 0$, $(n-k-1)^2+4k > 0$, and $2n-4> 0$ for $n\geq3$.  Thus this term is negative, and the function is concave down on the given interval.
\end{proof}

\bibliography{biblio}{}
\bibliographystyle{plain}
\end{document}